\documentclass[10pt,a4paper]{article}
\usepackage{amsmath}
\usepackage{amsfonts}
\usepackage{amssymb}
\usepackage{tikz}
\usetikzlibrary{arrows,backgrounds,decorations,positioning,fit,petri}
\usepackage{ntheorem}
\usepackage{color}
\usepackage{multirow}
\usepackage{abstract}
\usepackage{pdflscape}

\newtheorem{example}{Example}[section]

\newtheorem{remark}{Remark}[section]
\newtheorem{theorem}[example]{Theorem}

\newtheorem{definition}[example]{Definition}
\newtheorem{proposition}[example]{Proposition}

\newtheorem{lemma}[example]{Lemma}

\newenvironment{proof}{\begin{Proof}}{\qed\end{Proof}}

\def\hs{\hbox to 3mm{}}
\def\hhs{\hbox to 5cm{}}

\def\NN{\mathbb{N}}
\def\ss{\smallskip}

\def\H{\mathcal{H}}

\def\MQS{\mathbf{MQSym}}

\def\wmat{\mathrm{WMat}}

\def\qed{$\hfill\square$}

\title{\Large\bf\sc A word Hopf algebra based on the selection/quotient principle}
\author{\textsc{G\'erard H.E. Duchamp$^{1}$\footnote{gerard.duchamp@lipn.univ-paris13.fr}, Nguyen Hoang-Nghia$^{1}$
\footnote{nguyen.hoang@lipn.univ-paris13.fr}, 
Adrian Tanasa$^{1,2}$\footnote{adrian.tanasa@ens-lyon.org}}\\[2ex]  
}
\date{}

\begin{document}
\parindent=0pt
\maketitle

\begin{abstract}
\centering
    \begin{minipage}{0.8\textwidth}              
In this paper, we define a Hopf algebra structure on the vector space spanned by packed words using a selection/quotient coproduct. 
We show that this algebra is free on its irreducible packed words. 
We also construct the Hilbert series of this Hopf algebra and we investigate its primitive elements.
    \end{minipage}\let\thefootnote\relax\footnotetext{\textit{Date:} \today}
\end{abstract}

\newpage

\section{Introduction}

In algebraic combinatorics, one often associates algebraic structures with various sets of combinatorial objects. Such structures are, for example, Hopf algebras on trees, graphs, tableaux, matroids, words, etc

\medskip
A first type of combinatorial Hopf algebra structure is constructed using the selection/quotient principle. This simply means that the comultiplication is of the form
\begin{equation}\label{eq:type1} \Delta(S) = \sum_{\substack{A \subseteq S\\ +\, Conditions}} S[A] \otimes S/_A ,\end{equation} 
where $S[A]$ is a substructure of $S$ and $S/_A$ is a quotient.

Examples of such Hopf algebras are the Connes-Kreimer Hopf algebra of Feynman graphs, underlying the combinatorics of perturbative renormalization in quantum field theory \cite{Con-Kre2} or in non-commutative Moyal quantum field theory \cite{tanasa-kreimer}, \cite{tanasa-vignes-tourneret} (the interested reader may also refer to \cite{tanasa2}, \cite{Tan12} for some short reviews on these algebras). For the sake of completeness, let us also mention that similar Hopf algebraic structures have been proposed \cite{Markopoulou}, \cite{tanasa3}  for quantum gravity spin-foam models.

\medskip
Matroid theory has been introduced by Whitney in \cite{whitney}. A structure of matroid Hopf algebra is defined in \cite{schmi}, where the product is given by the direct sum operation and the coproduct is given by the selection/quotient principle mentioned above. 

\medskip
A second type of combinatorial Hopf algebra structure relies on the selection/complement principle. This means that the comultiplication is of the form
\begin{equation}\label{eq:type2}
\Delta(S) = \sum_{\substack{A \subseteq S\\ +\, Conditions}} S[A] \otimes [S-A].
\end{equation} 

\medskip
Examples of such Hopf algebras are the Loday-Ronco Hopf algebra of planar binary trees \cite{Lod-Ron2} or the Hopf algebra of matrix quasi-symmetric functions $\MQS$ \cite{NCSFVI}.

\medskip
It would thus be interesting to find Hopf algebraic structures on words with a comultiplication not of type \eqref{eq:type2}, but of type \eqref{eq:type1}. This is the issue we address in this paper.

\medskip
In this article, we introduce a new Hopf algebraic structure, that we call $\wmat$, on the set of packed words with product given by the shifted concatenation and coproduct given by such a selection/quotient principle. 
Each letter $x_j, j\geq 1$ can be seen 
as the infinite column vector with $1$ at $j^{th}$ place and other entries zero. $\wmat$ has the zero column, $x_0$, as a special element. 
In graph theory, it would correspond to the self-loop. Using the notion of irreducible element, we prove that $\wmat$ is free as an algebra. 



\section{Algebra structure}
\subsection{Definitions}

Let $X$ be an infinite totally ordered alphabet $\{x_i\}_{i \geq 0}$ and $X^*$ be the set of words with letters 
in the alphabet $X$. 

A word $w$ of length $n=|w|$ is a mapping $i\mapsto w[i]$ from $[1..|w|]$ to $X$. For a letter $x_i\in X$, the partial degree $|w|_{x_i}$ is the number of times the letter $x_i$ occurs in the word $w$. One has:
\begin{equation}
	|w|_{x_i}\ =\sum_{j=1}^{|w|} \delta_{w[j],x_i}.
\end{equation}
For a word $w\in X^*$, one defines the alphabet $Alph(w)$ as the set of its letters, while $IAlph(w)$ the set of indices in $Alph(w)$.
\begin{equation}
Alph(w)=\{x_i|\ |w|_{x_i}\not=0\}\ ;\ IAlph(w)=\{i\in \NN|\ |w|_{x_i}\not=0\}.
\end{equation}
The upper bound $sup(w)$ is the supremum of $IAlph(w)$, i. e. 
\begin{equation}
sup(w) = sup_{\,\mathbb{N}} (IAlph(w)).
\end{equation}
Note that $sup(1_{X^*}) = 0$. 

\ss
Let us define the substitution operators. Let $w=x_{i_1} \dots x_{i_m}$ and $\phi : IAlph(w) \longrightarrow \mathbb{N}$, with $\phi(0)=0$. One then has:
\begin{equation}\label{eq:subs}
S_\phi(x_{i_1} \dots x_{i_m}) = x_{\phi(i_1)} \dots x_{\phi(i_m)}.
\end{equation} 

Let us define the pack operator of a word $w$. Let $\{j_1 , \dots ,j_k \} = IAlph(w)\setminus \{0\}$ with $j_1 < j_2 < \dots < j_k$ and define
$\phi_w$ as 
\begin{equation}\label{eq:packedword}
\phi_w(i) = \begin{cases} m \mbox{ if } i = j_m \\ 0 \mbox{ if } i = 0 \end{cases}.
\end{equation}                                                 

The corresponding packed word, denoted by $pack(w)$, is $S_{\phi_w}(w)$. A word $w \in X^*$  is said to be \textit{packed} if $w = pack(w)$.

\begin{example}
Let $w = x_1 x_1 x_5 x_0 x_4$ . One then has $pack(w) = x_1 x_1 x_3 x_0 x_2$.
\end{example}

\begin{remark}
The presence of the letter $x_0$ dramatically influences the 
picture since one has an infinite number of distinct packed words of weight $m$ (the weight is, here, the sum of the indices), which are obtained by
adding multiple copies of the letter $x_0$.
\end{remark}

\begin{example}
The packed words of weight 2 are $x_0^{k_1} x_1 x_0^{k_2} x_1 x_0^{k_3}$, with $k_1, k_2, k_3 \geq 0$.
\end{example}

The operator $pack: X^* \longrightarrow X^*$ is idempotent ($pack \circ pack = pack$). It defines, by linear extension, a projector. The image, $pack(X^*)$, is the set of packed words.

\vskip.5cm
Let $u,v$ be two words; one defines the shifted concatenation $*$ by \begin{equation}u*v = uT_{sup(u)}(v),\end{equation} where, for $t \in \NN$, $T_t(w)$ denotes the image of $w$ by $S_\phi$ for $\phi(n) = n+t$ if $n>0$ and $\phi(0) = 0$ (in general, all letters can be reindexed except $x_0$). It is straightforward to check that, in the case the words are packed, the result of a shifted concatenation is a packed word.

\begin{definition}
Let $k$ be a field. One defines a vector space $\mathcal{H} = span_{k}(pack(X^*))$. One can endow this space with a product (on the words) given by 
 \begin{align*}\mu : &\mathcal{H} \otimes \mathcal{H} \longrightarrow \mathcal{H}, \\ & u \otimes v \longmapsto u*v.\end{align*}
\end{definition}

\begin{remark}
The product above is similar to the shifted concatenation for permutations (see \cite{NCSFVI}). 
Moreover, if $u,v$ are two words in $X^*$, then $sup(u*v) = sup(u) + sup(v)$. 
\end{remark}

\begin{proposition}
$(\mathcal{H},\mu,1_{X^*})$ is an associative algebra with unit (AAU).
\end{proposition}
\begin{proof}

Let $u,v,w$ be three words in $\mathcal{H}$. One then has:
\begin{align}
(u*v)*w  = (uT_{sup(u)}(v))T_{sup(u*v)}(w)  = u(T_{sup(u)}(v)(T_{sup(u) +sup(v)})(w)) = u* (v*w).
\end{align}

Thus, $(\mathcal{H},\mu)$ is associative. On the other hand, for all $u \in pack(X^*)$, one has:

\[u*1_{X^*}  = uT_{sup(u)}(1_{X^*}) = u1_{X^*} = u\ ,\]

\[1_{X^*}*u  = (1_{X^*})T_{sup(1_{X^*})}(u)  = (1_{X^*})u = u\ .\]

Now remark that $pack(1_{X^*})=1_{X^*}$. This is clear from the fact that $1_{X^*} = 1_\mathcal{H}$.

One concludes that $(\mathcal{H},\mu,1_{X^*})$ is an AAU.
\end{proof}

As already announced in the introduction, we call this algebra $\wmat$.

\ss
\begin{remark}
The product is non-commutative, for example: $x_1*x_1x_1 \neq x_1x_1*x_1$.
\end{remark}

Let $w = x_{k_1} \dots x_{k_n}$ be a word and $I \subseteq [1 \dots n]$. A sub-word $w[I]$ is defined as $x_{k_{i_1}} \dots x_{k_{i_l}}$, where
$i_j \in I$.

\ss
\begin{lemma}\label{lm:subword}Let $u,v$ be two words. Let $I \subset [1\dots |u|]$ and $J\subset [|u|+1\dots |u|+|v|]$. One then has \begin{equation}\label{eq:mor1}pack(u*v[I+J]) = pack(u[I])*pack(v[J']),\end{equation} where $J'$ is the set $\{i-|u|\}_{i\in J}$.
\end{lemma}

\begin{proof}
By direct computation, one has:

\begin{equation}\label{eq:prod}
 pack(u*v[I+J]) = pack(uT_{sup(u)}(v)[I+J]) = pack(u[I]T_{sup(u)}(v)[J]). 
\end{equation}

One further has $sup(u[I])\leq sup(u)$, and this, together with \eqref{eq:prod} leads to:
\begin{align}
pack(u[I]T_{sup(u)}(v)[J]) =  pack(u[I]T_{sup(u[I])}(v[J'])) = pack(u[I])*pack(v[J']). 
\end{align}
\end{proof}

\medskip
\begin{theorem}
Let $k<X>$ be equipped with the shifted concatenation. The mapping pack \begin{equation}
k<X> \xrightarrow{pack} \mathcal{H}
\end{equation} is then a morphism AAU.
\end{theorem}

\begin{proof}
By using the Lemma \ref{lm:subword} and taking $I=[1 \dots |u|]$ and $J = [|u|+1 \dots |u|+|v|]$, one gets the conclusion.
\end{proof}

\subsection{$\wmat$ is a free algebra}\label{sec:freealg}

$\wmat$ is, by construction, the algebra of the monoid $pack(X^*)$, therefore to check that $\wmat$ is a free algebra, it is sufficient to show that $pack(X^*)$ is a free monoid on its letters. In the following diagram, a free monoid is a pair $(F(X),j_X)$ where $F(X)$ is a monoid, $j_X : X \rightarrow F(X)$ is a mapping such that $(\forall M \in Mon)$ $(\forall f : X \rightarrow M)$ $(\exists ! f_X \in Mor(F(X),M))$ $f = f_x \circ j_X$.

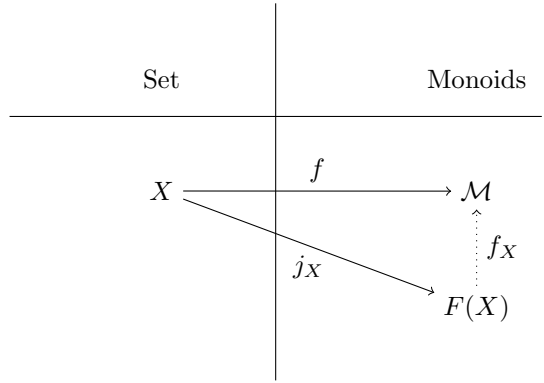
\begin{center}
\begin{tikzpicture}
\node (set) at (0,0) {Set};
\node (monoids) [right = 3cm of set] {Monoids};
\node (X) [below=1cm of set] {$X$};
\node (M) [below=1cm of monoids] {$\mathcal{M}$};
\node (F) [below=2.5cm of monoids] {$F(X)$};
\draw (5,-.5) to (-2,-.5);
\draw (1.5,1) to (1.5,-4);
\draw[->] (X) to node [above] {$f$} (M);
\draw[->] (X) to node[below] {$j_X$} (F);
\draw[->,dotted] (F) to node [right]{$f_X$} (M);
\end{tikzpicture}
\end{center}

Here we will use an ``internal" characterization of free monoids in terms of irreducible elements.

\begin{definition}\label{def:irrword}
A packed word $w$ in $pack(X^*)$ is called an irreducible word 
if and only if it cannot be written under the form $w = u*v$, 
 where $u$ and $v$ are two non trivial packed words.
\end{definition}

\begin{example}
The word $x_1x_1x_1$ is an irreducible word. The word $x_1x_1x_2$ is a reducible word because 
it can be written as $x_1x_1x_2 = x_1x_1 *x_1$.
\end{example}

\begin{proposition}
If $w$ is a packed word, then $w$ can be written uniquely as 
$w = v_1*v_2*\dots *v_n$, where $v_i$, $1\leq i \leq n$, are non-trivial irreducible words.
\end{proposition}

\begin{proof}
The $i^{th}$ position of word $w$ is called an admissible cut if $sup(w[1\dots i]) = inf(w[i+1 \dots |w|]) - 1$ or $sup(w[i+1 \dots |w|]) = 0$, where $inf(w)$ is infimum of $IAlph(w)$. 

Because the length of word is finite, one can get $w = v_1* v_2* \dots v_n$, 
with $n$ maximal and $v_i$ non trivial, $\forall 1\leq i \leq n$.

One assumes that one word can be written in two ways 
\begin{equation}\label{eq:irre1} w = v_1* v_2* \cdots * v_n\end{equation} 
and 
\begin{equation}\label{eq:irre2} w = v'_1* v'_2* \cdots * v'_m.\end{equation}

Denoting by $k$ the first number such that $v_k \neq v'_k$, without loss of generality, one can suppose that $|v_k|<|v'_k|$. From equation \eqref{eq:irre1}, the $k^{th}$ position is an admissible cut of $w$. From equation \eqref{eq:irre2}, the $k^{th}$ position is not an admissible cut of $w$. One thus has a contradiction. One has $n=m$ and $v_i=v'_i$ for all $1\leq i \leq n$.
\end{proof}

\medskip
One can thus conclude that $pack(X^*)$ is free as monoid with the packed words as a basis.



\section{Bialgebra structure}
  
Let us give the definition of the coproduct and prove that the coassociativity property holds.

\begin{definition}
Let $A \subset X$, one defines $w/A = S_{\phi_A}(w)$ with ${\phi_A}(i) = \begin{cases} i \mbox{ if } x_i \not\in A,\\ 0 \mbox{ if } x_i \in A\end{cases}.$

Let u be a word. One defines $^w/_u = ^w/_{Alph(u)}$.
\end{definition}

\begin{definition}\label{def:coprod}
The coproduct of $\H$ is given by \begin{equation}\label{eq:coprod}\Delta(w) = \sum_{I+J=[1\dots |w|]}pack(w[I]) \otimes pack(w[J]/_{w[I]}), \forall w \in \H, \end{equation} where this sum runs over all partitions of $[1 \dots |w|]$ divided into two blocks, $I \cup J = [1 \dots |w|]$ and $I \cap J =\emptyset$.
\end{definition}

\begin{example}
One has:
\begin{equation*}
\Delta(x_1x_2x_1) = x_1x_2x_1 \otimes 1_{X^\ast} + x_1 \otimes x_1x_0 + x_1 \otimes x_1^2 + x_1 \otimes x_0x_1 + x_1x_2 \otimes x_0 + x_1^2 \otimes x_1 + x_2x_1 \otimes x_0 + 1_{X^\ast}\otimes x_1x_2x_1.
\end{equation*}
\end{example}

\ss
Let us now prove the coassociativity.

\ss
Let $I=[i_1,\dots,i_n]$, and $I_1 =[j_1,\dots,j_k] \subseteq [1\dots |I|]$. Let $\alpha$ be a mapping: \begin{align}
\alpha:  I & \longrightarrow [1\dots n],\notag\\  i_s & \longmapsto s.
\end{align}

\begin{lemma}\label{lm:coprod1}
Let $w \in X^*$ be a word, $I$ be a subset of $[1\dots |w|]$ and $I_1 \subset [1 \dots |I|]$. 
One then has
 \begin{equation}\label{eq:coprod1} pack(w[I])[I_1] = S_{\phi_{w[I]}}(w[I'_1]),\end{equation} where $I'_1$ is $\alpha^{-1}(I_1)$ and $\phi_{w[I]}$ is the packing map of $w[I]$ that is given in \eqref{eq:packedword} .
\end{lemma}

\begin{proof}
Using the definition of packing map $\phi_{w[I]}$, one can directly check that equation \ref{eq:coprod1} holds.
\end{proof}

\ss
\begin{lemma}Let $w \in X^*$ be a word and $\phi$ be a strictly increasing map from $IAlph(w)$ to $\mathbb{N}$. One then has:
\begin{itemize}
\item[1)]  \begin{equation}pack(S_{\phi}(w)) = pack(w).\end{equation}
\item[2)] \begin{equation}\label{eq:quotionpack}S_{\phi}(^{w_1}/_{w_2})= ^{S_{\phi}(w_1)}/_{S_{\phi}(w_2)}.\end{equation}
\end{itemize}\label{lm:coprod2}
\end{lemma}

\begin{proof}

1) One has \begin{equation*}
pack(S_{\phi}(w))  = S_{\phi_0}(S_{\phi}(w))= S_{\phi_0\circ\phi}(w),
\end{equation*}
where $\phi_0$ is the packing map which is given in \eqref{eq:packedword}.

Note that both $\phi$ and $\phi_0$ are strictly increasing maps. 

Denote $I=IAlph(w) = \{j_1,j_2,\dots,j_k\}$, $j_1 < j_2 < \dots <j_k$, the image set $\phi(I) = \{j'_i,j'_i = \phi(j_i),i=1\cdots k \}$ one has $j'_1 < j'_2 < \dots < j'_k$. From the definition of $\phi_0$, one has: $\phi_0(j'_i) = i = \phi_0(j_i)$. This leads to:\begin{equation}
S_{\phi_0\circ\phi}(w) = S_{\phi_0}(w) = pack(w).
\end{equation}

2) Let $I_2 = Alph(w_2)$ and $I'_2 = Alph(S_{\phi}(w_1))$.

Let us rewrite the two sides of equation \eqref{eq:quotionpack}, the left hand side (LHS) and the right hand side (RHS):

\begin{equation}\label{eq:LHSquotionpack}
LHS = S_{\phi}(S_{\phi_{I_2}}(w_1)) = S_{\phi\circ\phi_{I_2}}(w_1),
\end{equation}

\begin{equation}\label{eq:RHSquotionpack}
RHS = S_{\phi_{I'_2}}(S_{\phi}(w_1)) = S_{\phi_{I'_2}\circ\phi}(w_1).
\end{equation}

With $x_i \in Alph(w_1)$, one has two cases:
\begin{itemize}
\item[1.] If $x_i \in I_2$, then $\phi_{I_2}(i)=0$ and $\phi\circ\phi_{I_2}(i)=0$ because $\phi$ is a strictly increasing map.

On the other hand, $\phi(i) \in I'_2$ and this implies $\phi_{I'_2}\circ\phi(i)=\phi_{I'_2}(\phi(i))=0$.
\item[2.] If $x_i \not\in I_2$, then $\phi_{I_2}(i)=i$ and $\phi\circ\phi_{I_2}(i)=\phi(i)$.

On the other hand, because $\phi$ is a strictly increasing map, then $\phi(i) \not\in I'_2$, and $\phi_{I'_2}\circ\phi(i) = \phi(i)$.
\end{itemize}

One thus has $\phi\circ\phi_{I_2}(i) = \phi_{I'_2}\circ\phi(i)$. Using this result and the two equations above \eqref{eq:LHSquotionpack} and \eqref{eq:RHSquotionpack}, we conclude the proof.
\end{proof}

\ss
\begin{lemma} Let $w$ be a word in $\mathcal{H}$, and $I,J,K$ be three disjoint subsets of $\{1 \dots |w|\}$. One then has: \begin{equation}
\frac{^{w[K]}/_{w[I]}}{^{w[J]}/_{w[I]}} = ^{w[K]}/_{w[I+J]}.
\end{equation}\label{lm:coprod3}
\end{lemma}

\begin{proof}

Using Lemma \ref{lm:coprod2}, one has: \begin{align}
\frac{^{w[K]}/_{w[I]}}{^{w[J]}/_{w[I]}} = ^{S_{\phi_I}(w[K])}/_{S_{\phi_I}(w[K])} = S_{\phi_I}(^{w[K]}/_{w[J]})  = S_{\phi_I}(S_{\phi_J}(w[K])) = S_{\phi_I \circ \phi_J}(w[K]) =  ^{w[K]}/_{w[I+J]}.
\end{align}
\end{proof}

\ss
\begin{proposition}
The vector space $\mathcal{H}$ endowed with the coproduct \eqref{eq:coprod} is a coassociative coalgebra with co-unit (c-AAU). The co-unit is given by: 
$$\epsilon(w) = \begin{cases} 1 \mbox{ if } w = 1_\H,\\0 \mbox{ otherwise}.  \end{cases}$$ 
\end{proposition}

\begin{proof}

Let us first prove the coassociativity of the coproduct \eqref{eq:coprod}, namely \begin{equation}\label{eq:coass}
(\Delta \otimes Id)\circ \Delta(w) = (Id \otimes \Delta)\circ\Delta(w).
\end{equation}

The LHS of the coassociativity condition \eqref{eq:coass} can be written:

\begin{align}\label{eq:LHSassc}
& (\Delta \otimes Id)\circ \Delta(w) =  \sum_{I+J = [1\dots |w|]} \left(\sum_{I_1+I_2 = [1 \dots |I|]} pack(pack(w[I])[I_1]) \otimes  pack(^{pack(w[I][I_2])}/_{pack(w[I][I_1])})\right)\notag\\
  & \otimes pack(^{w[J]}/_{w[I]}) = \sum_{I+J = [1\dots |w|]} \left(\sum_{I'_1+I'_2 = I} pack(S_{\phi}(w[I'_1])) \otimes  pack(^{S_{\phi}(w[I'_2])}/_{S_{\phi}(w[I'_1])}) \right)\notag\\ & \otimes pack(^{w[J]}/_{w[I]})
 =  \sum_{I'_1+I'_2+J = [1\dots |w|]} pack(w[I'_1]) \otimes pack(^{w[I_2]}/_{w[I'_1]})  \otimes pack(^{w[J]}/_{w[I_1+I_2]}).
\end{align}

The RHS of the coassociativity condition \eqref{eq:coass} can be written as:

\begin{align}\label{eq:RHSassc}
& (Id \otimes \Delta)\circ \Delta(w) = 
\sum_{I+J=[1\dots |w|]} pack(w[I])\otimes \left(\sum_{J_1+J_2=[1\dots |J|]} pack(pack(^{w[J]}/_{w[I]}))[J_1]) \right. \notag\\ & \left. \otimes pack(^{pack(^{w[J]}/_{w[I]})[J_2]}/_{pack(^{w[J]}/_{w[I]})[J_1]}))\right) 
 = \sum_{I+J=[1\dots |w|]} pack(w[I])\otimes \left(\sum_{J'_1+J'_2=J} pack(^{w[J'_1]}/_{w[I]}) \right. \notag\\ & \left. \otimes pack(S_{\phi}(^{^{w[J'_2]}/_{w[I]}}/_{{w[J'_1]}/_{w[I]}}))\right)
 = \sum_{I+J=[1\dots |w|]} pack(w[I])\otimes \left(\sum_{J'_1+J'_2=J} pack(^{w[J'_1]}/_{w[I]})\right. \notag\\ & \left. \otimes pack(^{w[J'_2]}/_{w_{[I+J'_1]}})\right)
 = \sum_{I+J'_1+J'_2=[1\dots |w|]} pack(w[I])\otimes pack(^{w[J'_1]}/_{w[I]}) \otimes pack(^{w[J'_2]}/_{w_{[I+J'_1]}}).
\end{align}

Using the two equation \eqref{eq:LHSassc} and \eqref{eq:RHSassc}, one can conclude that the coproduct \eqref{eq:coprod} is coassociative.

Let us now prove the following  \begin{equation}\label{eq:counit}(\epsilon \otimes Id)\circ \Delta(w) = (Id \otimes \epsilon)\circ\Delta(w), \end{equation} for all word $w \in \mathcal{H}$.

Let us rewrite the LHS and the RHS of the equation \eqref{eq:counit}:

\begin{align}
LHS & =(\epsilon \otimes Id)\left(\sum_{I+J=[1\dots |w|]}pack(w[I]) \otimes pack(w[J]/w[I]) \right)\notag\\
& = \sum_{I+J=[1\dots |w|]}\epsilon(pack(w[I])) \otimes pack(w[J]/w[I]) = 1_\H \otimes pack(w) = pack(w).
\end{align}

\begin{align}
RHS & =(Id \otimes \epsilon)\left(\sum_{I+J=[1\dots |w|]}pack(w[I]) \otimes pack(w[J]/w[I]) \right)\notag\\
& = \sum_{I+J=[1\dots |w|]}pack(w[I]) \otimes \epsilon(pack(w[J]/w[I])) = pack(w) \otimes 1_\H = pack(w).
\end{align}

One thus concludes that $(\mathcal{H},\Delta,\epsilon)$ is a c-AAU.
\end{proof}

\ss
\begin{remark}
This coalgebra is not cocommutative, for example: 
\begin{align*}
T_{12} \circ \Delta (x_1^2) & = T_{12} (x_1^2 \otimes 1_\H + 2 x_1 \otimes x_0 + 1_\H \otimes x_1^2) \\ & = x_1^2 \otimes 1_\H + 2 x_0 \otimes x_1 + 1_\H \otimes x_1^2 \neq \Delta(x_1^2),
\end{align*}
where the operator $T_{12}$ is given by $T_{12}(u\otimes v) = v \otimes u$.
\end{remark}

\begin{lemma}
Let $u,v$ be two words. Let $I_1 +J_1 = [1\dots |u|]$ and $I_2 + J_2 = [|u|+1\dots |u|+|v|]$. One then has
\begin{equation}\label{eq:mor2}
pack(^{u*v[J_1+J_2]}/_{u*v[I_1+I_2]}) = pack(^{u[J_1]}/_{u[I_1]}) * pack(^{v[J'_2]}/_{v[I'_2]}),
\end{equation} where $I'_2$ is the set $\{k-|u|,k\in I_2\}$ and $J'_2$ is the set $\{k-|u|,k\in J_2\}$.
\end{lemma}

\begin{proof}
One has:
\begin{align}
& pack(^{u*v[J_1+J_2]}/_{u*v[I_1+I_2]}) = pack(S_{\phi_{I_1+I_2}}(u*v[J_1+J_2])) = pack(S_{\phi_{I_1+I_2}}(uT_{sup(u)}(v)[J_1+J_2])) \notag\\
 =& pack(S_{\phi_{I_1}+\phi_{I_2}}(u[J_1]T_{sup(u)}(v)[J_2]))= pack(S_{\phi_{I_1}}S_{\phi_{I_2}}(u[J_1]T_{sup(u)}(v[J'_2]))) \notag\\
 =& pack(S_{\phi_{I_1}}(u[J_1])S_{\phi_{I_2}}(T_{sup(u)}(v[J'_2])))= pack(^{u[J_1]}/_{u[I_1]})T_{sup(^{u[J_1]}/_{u[I_1]})}pack(S_{\phi_{I'_2}}(v[J'_2])) \notag\\
 =& pack((^{u[J_1]}/_{u[I_1]})*pack(^{u[J'_2]}/_{u[I'_2]}).
\end{align}
\end{proof}

\ss
\begin{proposition}
Let $u,v$ be two words in $\mathcal{H}$. One has: \begin{equation} \Delta(u*v) = \Delta(u)*^{\otimes 2} \Delta(v).\end{equation}
\end{proposition}

\begin{proof}

One has:

 \begin{align}
& \Delta(u*v)  
= \sum_{\substack{I_1 + I_2 = I \\ J_1 + J_2 = J \\I_1,J_1 \subset [1\dots |u|]\\I_2,J_2 \subset [|u|+1\dots |u|+|v|]}} \left(pack(u * v[I_1+I_2])\right)  \otimes\left( pack(^{(u*v[J_1+J_2]}/_{(u*v[I_1+I_2]})\right)\notag\\
=&\sum_{\substack{I_1+J_1 = [1\dots |u|]\\I'_2+J'_2 = [1\dots |v|]}} \left(pack(u[I_1])\otimes pack(^{u[J_1]}/_{u[I_1]})\right) *\left(pack( v[I_2]) \otimes pack(^{u[J'_2]}/_{u[I'_2]})\right)\notag\\
=&\left(\sum_{I_1+J_1 = [1\dots |u|]} pack(u[I_1])\otimes pack(^{u[J_1]}/_{u[I_1]})\right)  *\left(\sum_{I'_2+J'_2 = [1\dots |v|]} pack( v[I_2]) \otimes pack(^{u[J'_2]}/_{u[I'_2]})\right)\notag\\ 
=& \Delta(u) *^{\otimes 2} \Delta(v).
\end{align}
\end{proof}

\medskip
Since $\mathcal{H}$ is graded by the word's length, one has the following theorem:

\begin{theorem}
$(\mathcal{H},*,1_\H,\Delta,\epsilon)$ is a Hopf algebra.
\end{theorem}

\begin{proof}

The proof follows from the above results.
\end{proof}

\medskip
For $w\neq 1_\H$, the antipode is given by the recursion: \begin{equation}
S(w) = -w - \sum_{I+J =[1\dots |w|], I,J\neq \emptyset} S(pack(w[I]))*pack(^{w[J]}/_{w[I]}).
\end{equation}



\section{Hilbert series of the Hopf algebra $\wmat$}

In this section, we compute the number of packed words with length $n$ and supremum $k$. It is the same as the number of cyclically ordered partitions of an $n$-element set. Using the formula of Stirling numbers of the second kind (see \cite{Co74}), one can get the explicit formula for the number of packed words with length $n$, number which we denote by $d_n$.  

\begin{definition} The Stirling numbers of the second kind count the number of set partitions of an $n$-element set into precisely $k$ non-void parts. The Stirling numbers, denoted by $S(n,k)$ are given by the recursive definition:

\begin{enumerate}
\item $S(n,n) =1 (n \geq 0)$,
\item $S(n,0) =0 (n > 0)$,
\item $S(n+1,k) = S(n,k-1) + kS(n,k)$, for $0<k\leq n$.
\end{enumerate}
\end{definition}

One can define a word without $x_0$ by its positions, 
this means that if a word $w=x_{i_1}x_{i_2}\dots x_{i_n}$ has length $n$ and 
alphabet $IAlph(w) = \{1,2,\dots,k\}$, then this word can be determine from the list 
$[S_1,S_2,\dots,S_k]$, where $S_i$ is the set of positions of $x_i$ in the word $w$, with $1 \leq i \leq k$. It is straightforward to check that $(S_i)_{0\leq i \leq k}$ is a partition of $[1 \dots n]$.

One can divide the set of packed words with length $n$ and supremum $k$ in two parts: \textit{``pure"} packed words (which have no $x_0$ in their alphabet), denote $pack_{n,s}^+(X)$ and packed words which have $x_0$ in their alphabet, denote $pack_{n,k}^0(X)$. It is clear that: \begin{equation}
d(n,k) = \# pack_{n,k}^+(X) + \# pack_{n,k}^0(X).
\end{equation}

Let us now compute the cardinal of these two sets $pack_{n,k}^+(X)$ and $pack_{n,k}^0(X)$.

Consider a word $w \in pack_{n,k}^+(X)$, then $IAlph(w) = \{1,2,\dots,k\}$. This word is determined by $[S_1,S_2,\dots,S_k]$, in which $S_i$ is a set of positions of $x_i$, for $1\leq i \leq k$. One can see that:\begin{enumerate}
\item $S_i \neq \emptyset$, $\forall i \in [1,k]$;
\item $\sqcup_{1 \leq i \leq k} S_i = \{1,2,\dots,n\}$.
\end{enumerate}

Note that 1-2 hold even with $w = 1_\H$.

Thus, one has the cardinal of packed words with length $n$ and supremum $k$:
\begin{equation}\label{eq:d+}
d^+(n,k)=\# pack_{n,k}^+(X) = S(n,k)k!.
\end{equation}

Similarly, a 
word $w \in \# pack_{n,k}^0(X)$ can be determined by $[S_0,S_1,S_2,\dots,S_k]$ where $S_i$ is the set of positions of $x_i$, for all $0\leq i \leq k$. 
One then has: 
\begin{equation}\label{eq:d0}
d^0(n,k)=\# pack_{n,k}^0(X) = S(n,k+1)(k+1)!.
\end{equation}

From the two equations above, one can get the number of packed words with length $n$, supremum $k$:
\begin{equation}\label{eq:dnk1}
d(n,k) = d^+(n,k) + d^0(n,k) = S(n,k)k! + S(n,k+1)(k+1)! = S(n+1,k+1)k!.
\end{equation}

From this formula, using Maple, one can get some values of $d(n,k)$. We give in the Table \ref{tab:dnk} the first values.
\begin{table}[!ht]
\begin{center}
\begin{tabular}{ccrrrrrrrrr}\hline
 &&\multicolumn{9}{c}{k}\\
&&0&1&2&3&4&5&6&7&8\\\cline{1-11}
\multicolumn{1}{l}{\multirow{9}{*}{n}}& \multicolumn{1}{l}{0}&1&0&0&0&0&0&0&0&0\\
\multicolumn{1}{l}{}&\multicolumn{1}{l}{1}&1&1&0&0&0&0&0&0&0\\
\multicolumn{1}{l}{}&\multicolumn{1}{l}{2}&1&3&2&0&0&0&0&0&0\\
\multicolumn{1}{l}{}&\multicolumn{1}{l}{3}&1&7&12&6&0&0&0&0&0\\
\multicolumn{1}{l}{}&\multicolumn{1}{l}{4}&1&15&50&60&24&0&0&0&0\\
\multicolumn{1}{l}{}&\multicolumn{1}{l}{5}&1&31&180&390&360&120&0&0&0\\
\multicolumn{1}{l}{}&\multicolumn{1}{l}{6}&1&63&602&2100&3360&2520&720&0&0\\
\multicolumn{1}{l}{}&\multicolumn{1}{l}{7}&1&127&1932&10206&25200&31920&20160&5040&0\\
\multicolumn{1}{l}{}&\multicolumn{1}{l}{8}&1&255&6050&46620&166824&317520&332640&181440&40320\\\cline{1-11}
\end{tabular}
\caption{Values of $d(n,k)$ given by the explicit formula \eqref{eq:dnk1} and computed with Maple.\label{tab:dnk}}
\end{center}
\end{table}

Note the values of  Table \ref{tab:dnk} correspond to those of the triangular array $A028246$ of Sloane \cite{sloane}.

\begin{remark}
 Formulas \eqref{eq:d+} and \eqref{eq:d0} imply that the 
 packed words of length $n$ and supremum $k$ without, and respectively with, $x_0$ 
 are in bijection with the ordered partitions of $[n]$ in $k$ parts and respectively in $k+1$ parts. Therefore formula \eqref{eq:dnk1} implies that the set of packed words of length $n$ with supremum $k$ is in bijection with the circularly ordered partitions of $n+1$ elements in $k+1$ parts.
 \end{remark}

\ss
The formula for the number of packed words of length $n$, $d_n$ ($n\geq 1$), is then given by \begin{align}\label{eq:dn}
&d_n =\sum_{k=0}^{n} d(n,k)= \sum_{k=0}^n S(n+1,k+1)k!.
\end{align}

Using again Maple, one can get the values listed in Table \ref{tab:dn}.
\begin{table}[!ht]
\begin{center}
\begin{tabular}{crrrrrrrrrrrr|}\hline
n&0&1&2&3&4&5&6&7&8&9&10\\\hline
$d_n$&1&2&6&26&150&1082&9366&94586&1091670&14174522&204495126\\\hline
\end{tabular}
\caption{Some value of $d_n$ by the formula \eqref{eq:dn}.\label{tab:dn}}
\end{center}
\end{table}

The number of packed words is the sequence 
$A000629$ of Sloane
\cite{sloane}, where it is also mentioned that this sequence corresponds to 
the ordered Bell numbers sequence times two (except for the $0$th order term).

The ordinary and exponential generating function of our sequence are also given in \cite{sloane}. 
The ordinary one is given by 
the formula: $\sum_{n\ge 0} \frac{2^n n! x^n}{\prod_{k=0}^n (1+k x)}$. 
The exponential one 
is given by: $\frac{e^x}{2-e^x}$. Let us give the proof of this.

\ss
Firstly, recall that 
the exponential generating function of the ordered Bell numbers 
(see, for example, page $109$ of Philippe Flajolet's book \cite{flaj}) is:

\begin{equation}\label{eq:pt2}
\frac{1}{2-e^x} = \sum_{n\geq 0 } \sum_{k = 0}^n S(n,k)k! \frac{x^n}{n!}.
\end{equation}

By deriving both side of equation \eqref{eq:pt2} with respect to $x$ , one obtains: 
\begin{equation}\label{eq:pt3}
\frac{e^x}{2-e^x} = \sum_{n\geq 1 } \sum_{k = 1}^n S(n,k)k! \frac{x^{n-1}}{(n-1)!}.
\end{equation}

From equations \eqref{eq:dn} and \eqref{eq:pt3}, one gets the exponential generating function of our sequence:

\begin{equation}\label{eq:egfdn}
\frac{e^x}{2-e^x} = \sum_{n\geq 0 } \sum_{k = 0}^n S(n+1,k+1)k! \frac{x^{n}}{n!} = \sum_{n\geq 0 } d_n \frac{x^{n}}{n!}.
\end{equation}

\bigskip
Let us now investigate the combinatorics of irreducible packed words (see Definition \ref{def:irrword}).
Firstly, we notice that one still has an infinity of irreducible packed words of weight $m$, which are again
obtained by adding multiple copies of the letter $x_0$.

\begin{example}
The word $x_1x_0^kx_1x_0^kx_1$ (with $k$ an arbitrary integer) 
is an irreducible packed word of weight $3$.
\end{example}

Let us denote by $i_n$ the number of irreducible packed words of length $n$. Then one has:
\begin{equation}
i_n =  \sum_{\substack{j_1+ \dots + j_k =n\\j_l \neq 0}}(-1)^{k+1} d_{j_1} \dots d_{j_k}.
\end{equation}

Using Maple, one can get the values of $i_n$, which we give  in Table \ref{tab:irr} below.

\begin{table}[!ht]
\begin{center}
\begin{tabular}{crrrrrrrrrrrr}\hline
n&0&1&2&3&4&5&6&7&8&9&10\\\hline
$i_n$ &1 &2 &2 &10 &66 &538 &5170 &59906 &704226 &9671930 &145992338 \\\hline
\end{tabular}
\caption{Ten first values of the number of irreducible packed words.\label{tab:irr}}
\end{center}
\end{table}

Note that this sequence does not appear in Sloane's On-Line Encyclopedia of Integer Sequences \cite{sloane}.

\section{Primitive elements of $\wmat$}

Let us emphasize that this Hopf algebra, although graded, is not cocommutative and thus the primitive
elements do not generate the whole algebra but only the sub Hopf algebra on which $\Delta$ is cocommutative (the greatest subalgebra on which CQMM theorem holds).

We denote by $Prim(\wmat)$ the algebra generated by the primitive elements of $\H$.

\ss
Let us recall the following result:
\begin{lemma}\label{lm:graded}
Let $V^{(1)}$ and $V^{(2)}$ be two graded vector space.
\begin{equation}
V^{(i)} = \oplus_{n\geq 0} V_n^{(i)} ,\ i = 1, 2.
\end{equation}                                           
Let $\phi \in  Hom^{gr} (V^{(1)} , V^{(2)} )$, that means $(\forall n \geq 0)(\phi (V^{(1)}_n ) \subseteq V_n^{(2)})$.
Then, $Ker(\phi)$ is graded.
\end{lemma}

\medskip
One then has:
\begin{proposition}
$Prim(\wmat)$ is a Lie subalgebra of $\wmat$, graded by the word's length.
\end{proposition}

\begin{proof}
Let us define the mapping \begin{eqnarray}
\Delta^+:  & \wmat  \longrightarrow \wmat \otimes \wmat \notag\\
& \begin{cases} 1_\H & \longmapsto 0 \\ h & \longmapsto \Delta(h) - 1_\H \otimes h -h \otimes 1_\H \end{cases}.
\end{eqnarray}

This mapping is graded. Using Lemma \ref{lm:graded}, one has $Prim(\wmat) = Ker(\Delta^+)$. Thus, the subalgebra $Prim(\wmat)$ is graded.
\end{proof}

\medskip


Let us now compute the dimensions of the first few spaces $Prim(\wmat)_n$:
\begin{itemize}
 \item 
For $n=1$, 
one has a basis formed by the primitive elements $x_0$ and $x_1$. 
Then one can check that the primitive elements of length $1$ have the form $ax_0 + bx_1$, with $a$ 
and $b$ scalars. 

\item
For $n=2$, one has a basis formed by the primitive elements: $x_0 x_1 - x_1 x_0$ and $x_1 x_2 - x_2 x_1$. 
Then one can check that all the primitive elements of the length $2$ 
have the form $a(x_0 x_1 - x_1 x_0 ) + b(x_1 x_2 - x_2 x_1)$, with $a$ and $b$ scalars.
This comes from explicitly solving a system of $4$ equations with $d_2=6$ variables.

\end{itemize}

Nevertheless, the explicit calculations quickly become lengthy. Thus, for $n=3$, 
one has to solve a system of $22$ equations with $26$ variables. 



\bigskip     
Finally, let us give some a posteriori explanations on the choice of the $\wmat$ 
name for our algebraic structure.
We have defined here a Hopf algebra on some set of words with a selection/quotient coproduct rule in the spirit of graphs  \cite{Con-Kre2} and matroids Hopf algebras \cite{schmi} . We could thus call our algebra $\mathrm{WGraph}$ or $\wmat$, but we prefer the name $\wmat$ since matroids are more general structures than graphs.

On the other hand, the "W" in our Hopf algebra name simply means that we have a Hopf algebra with basis indexed by a subset $X$ of the free
monoids $\NN^\ast$. Let us state that, at this point, there is no polynomial realization of $\wmat$ 
(as is is the case for $WSym$ or $WQSym$ or the Connes-Kreimer Hopf algebras on trees (see \cite{pol-rel} and references within).

This actually seems to us to be an important perspective for future work related to the 
new Hopf algebra proposed in this paper.
 
\section*{Acknowledgements}
We acknowledge Jean-Yves Thibon for
various discussions and suggestions.
The authors also acknowledge a Univ. Paris 13, Sorbonne Paris Cit\'e BQR grant.
A. Tanasa further acknowledges
the grants PN 09 37 01 02 and CNCSIS Tinere Echipe 77/04.08.2010.
G. H. E. Duchamp acknowledges the grants ANR BLAN08-2\_332204 (Physique Combinatoire) and PAN-CNRS 177494 (Combinatorial Structures and Probability Amplitudes).



\addcontentsline{toc}{section}{References}
\bibliographystyle{plain}


{\small $^{1}$LIPN - Institut Galil\'ee - Universit\'e Paris 13, Sorbonne Paris Cit\'e, CNRS UMR 7030,} \\
{\small 99 avenue Jean-Baptiste Cl\'ement, F-93430 Villetaneuse, France,\\
$^{2}$Horia Hulubei National Institute for Physics and Nuclear Engineering,\\
P.O.B. MG-6, 077125 Magurele, Romania}

\end{document}